\documentclass[11pt]{article}
\usepackage{amsmath,amssymb}
\usepackage{cases}
\usepackage{graphicx}
\newtheorem{theorem}{Theorem}

\newtheorem{lemma}[theorem]{Lemma}
\newtheorem{remark}[theorem]{Remark}
\newcommand{\Ref}[1]{\mbox{(\ref{#1})}}
\newcommand{\bull}{\vrule height 1.8ex width 1.0ex depth 0ex}
\newenvironment{proof}{{\bf Proof : }}{\hfill$\bull$\medskip}
\newcommand{\Dev}[2]{{\frac{\partial {#1}}{\partial {#2}}}}
\newcommand{\AEW}{\textrm{a.e.}~~}
\newcommand{\Norm}[2]{\left\|#1\right\|_{#2}}
\newcommand{\Inner}[3]{\left<#1,#2\right>_{#3}}
\newcommand{\Set}[2]{{\left\{{#1}~;~{#2}\right\}}}
\newcommand{\Cinv}{C_{\textrm{inv}}}
\title{Constructive error analysis of a full-discrete finite element method for the heat equation}
\author{
Kouji Hashimoto$^{1}$, Takuma Kimura$^{2}$, Teruya Minamoto$^{2}$, Mitsuhiro T. Nakao$^{3}$
}
\date{\small\sl 
hashimot@nakamura-u.ac.jp\\
$^{1}$Division of Infant Education, Nakamura Gakuen Junior College, Fukuoka 814-0198, Japan\\
$^{2}$Department of Information Science, Saga University, Saga 840-8502, Japan\\
$^{3}$Faculty of Science and Engineering, Waseda University, Tokyo 169-8555, Japan
}
\begin{document}
\maketitle
\begin{abstract}
In this paper, we present a new full-discrete finite element method for the heat equation, and show the numerical stability of the method by verified computations.
Since, in the error analysis, we use the constructive error estimates proposed ny Nakao et. all in 2013, this work is considered as an extention of that paper.
We emphasize that concerned scheme seems to be a quite normal Galerkin method and easy to implement for evolutionary equations comparing with previous one.
In the constructive error estimates, we effectively use the numerical computations with guaranteed accuracy.
\end{abstract}

\section{Introduction}
% ---------------------------------------
The purpose of this paper is to establish the constructive a priori error estimates for a full-discrete approximations $Q_h^k u$, which is defined in this section, of the solution $u$ to the following linear heat equation:
\begin{eqnarray}\label{parabolic}
\begin{array}{rclcl}
\Dev{}{t}u-\nu\Delta u&=&f&{\rm in}&\Omega\times J,\\
u(x,t)&=&0&{\rm on}&\partial\Omega\times J,\\
u(0)&=&0&{\rm in}&\Omega.
\end{array}
\end{eqnarray}
Here, $\Omega \subset \mathbb{R}^d$ $(d \in \{1,2,3\})$ is a bounded polygonal or polyhedral domain; $J:=(0,T) \subset \mathbb{R}$, (for a fixed $ T<\infty$) is a bounded open interval; the diffusion coefficient $\nu$ is a positive constant; and $f \in L^2\bigl(J;L^2(\Omega)\bigr)$, where, in general for any normed space $Y$,  we define the time-dependent Lebesgue space $L^2\bigl(J;Y \bigr)$ as a space of square integrable $Y$-valued functions on $J$. That is,   
\[
f \in L^2\bigl(J;Y \bigr) \quad \Leftrightarrow \quad \int_J ||f(t)||_Y^2 dt < \infty. 
\]
In the discussion below, we refer to the a priori estimates as {\it `constructive'} if all the constants can be numerically determined. 
% ---------------------------------------
\subsection{Notations}
% ---------------------------------------
The notations to the spaces in this paper are very similar to that presented in \cite{Nakao-SIAM}, we include these here for the sake of convenience.

We denote $L^2(\Omega)$ and $H^1(\Omega)$ as the usual Lebesgue and the first order $L^2$-Sobolev spaces on $\Omega$, respectively, and by $\Inner{u}{v}{L^2(\Omega)}:=\int_\Omega u(x)v(x)\,dx$ the natural inner product for $u$, $v\in L^2(\Omega)$. 
By considering the boundary and initial conditions, we define the following subspaces of $H^1(\Omega)$ and $H^1(J)$ as
\[ H_0^1(\Omega):=\Set{u \in H^1(\Omega)}{u = 0 \textrm{ on } \partial\Omega} \quad \mbox{and} \quad V^1(J):=\Set{u \in H^1(J)}{u(0) = 0}, \]
respectively. 
These are Hilbert spaces with inner products
\[ \Inner{u}{v}{H_0^1(\Omega)}:=\Inner{\nabla u}{\nabla v}{L^2(\Omega)^d}\quad \mbox{and} \quad \Inner{u}{v}{V^1(J)}:=\Inner{\Dev{u}{t}}{\Dev{v}{t}}{L^2(J)}. \]
Let $X(\Omega)$ be a subspace of $L^2(\Omega)$ defined by $X(\Omega):=\Set{u \in L^2(\Omega)}{\triangle u \in L^2(\Omega)}$. We define 
\[
V^1\bigl(J;L^2(\Omega)\bigr):=\Set{u \in L^2\bigl(J;L^2(\Omega)\bigr)}{\Dev{u}{t} \in L^2\bigl(J;L^2(\Omega)\bigr)\quad \mbox{and} \quad u(x,0)=0 \textrm{ in } \Omega},
\] 
with inner product $\Inner{u}{v}{V^1\bigl(J;L^2(\Omega)\bigr)}:=\Inner{\Dev{u}{t}}{\Dev{v}{t}}{L^2\bigl(J;L^2(\Omega)\bigr)}$.
In the following discussion, abbreviations like $L^2H_0^1$ for $L^2\bigl(J;H_0^1(\Omega)\bigr)$ will often be used. 
We set $V(\Omega,J):= V^1\bigl(J;L^2(\Omega)\bigr) \cap L^2\bigl(J;H_0^1(\Omega)\bigr)$.
Moreover, we denote the partial differential operator $\triangle_t:V(\Omega,J) \cap L^2\bigl(J;X(\Omega)\bigr) \to L^2\bigl(J;L^2(\Omega)\bigr)$ by $\triangle_t:=\Dev{}{t}-\nu\triangle$.

Now let $S_h(\Omega)$ be a finite-dimensional subspace of $H_0^1(\Omega)$ dependent on the parameter $h$.
For example, $S_h(\Omega)$ is considered to be a finite element space with mesh size $h$.
Let $n$ be the degree of freedom for $S_h(\Omega)$, and let $\{\phi_i\}_{i=1}^n \subset H_0^1(\Omega)$ be the basis of $S_h(\Omega)$.
Similarly, let $V_k^1(J)$ be an approximation subspace of $V^1(J)$ dependent on the parameter $k$. 
Let $m$ be the degree of freedom for $V_k^1(J)$, and let $\{\psi_i\}_{i=1}^m \subset V_k^1(J)$ be the basis of $V_k^1(J)$.
Let $V^1\bigl(J;S_h(\Omega)\bigr)$ be a subspace of $V(\Omega,J)$ corresponding to the semidiscretized approximation in the spatial direction.
We define the $H_0^1$-projection $P_h^1 u \in  S_h(\Omega)$ of any element $u \in H_0^1(\Omega)$ by the following variational equation:
\begin{align}
  \Inner{\nabla (u-P_h^1u)}{\nabla v_h}{L^2(\Omega)^d} = 0, \quad \forall v_h \in S_h(\Omega). \label{eq:IPP:Ph1 def}
\end{align}
Similarly,  for any element $u \in V^1(J)$, the $V^1$-projection $P^k_1:V^1(J) \to V_k^1(J)$ is defined by follows:
\begin{eqnarray*}
  \Inner{\Dev{}{t}(u-P_1^ku)}{\Dev{}{t} v_k}{L^2(J)} = 0, \quad \forall v_k \in V_k^1(J).
\end{eqnarray*}

Now let $\Pi_k:V^1(J) \to V_k^1(J)$ be an interpolation operator.
Namely, if the nodal points of $J$ are given by $0=t_0<t_1<\cdots<t_m=T$, then for an arbitrary $u \in V^1(J)$, the interpolation $\Pi_ku$ is defined as the function in $V_k^1(J)$ satisfying:
\begin{align}
  u(t_i) = \bigl(\Pi_ku\bigr)(t_i), \quad \forall i \in \{1,\ldots,m\}. \label{eq:IPP:Pik def}
\end{align}
We know that there exist constants $C_\Omega(h)>0$, $C_J(k)>0$ and $\Cinv(h)>0$ satisfying
\begin{eqnarray*}
  \Norm{u-P_h^1u}{H_0^1(\Omega)}
    &\leq &C_\Omega(h)\Norm{\triangle u}{L^2(\Omega)}, \quad \forall u\in H_0^1(\Omega) \cap X(\Omega), \\
  \Norm{u-\Pi_ku}{L^2(J)}
    &\leq& C_J(k)\Norm{u}{V^1(J)}, \quad \forall u\in V^1(J),\\
  \Norm{u_h}{H_0^1(\Omega)}
    &\leq& \Cinv(h)\Norm{u_h}{L^2(\Omega)}, \quad \forall u_h\in S_h(\Omega).
\end{eqnarray*}
Moreover, there exists a Poincar\'e constant $C_p>0$  satisfying
\begin{eqnarray*}
  \Norm{u}{L^2(\Omega)}
    &\leq& C_p\Norm{u}{H^1_0(\Omega)}, \quad \forall u\in H^1_0(\Omega).
\end{eqnarray*}
For example, if $\Omega$ is a bounded open rectangular domain in $\mathbb{R}^d$, and $S_h(\Omega)$ is the piecewise linear (P1) finite element space, then they can be taken by $C_\Omega(h)=\frac{h}{\pi}$ (see, e.g., \cite{Nakao 1998 best}) and $\Cinv(h)=\frac{\sqrt{12}}{h_{\min}}$, where $h_{\min}$ is the minimum mesh size for $\Omega$ (see, e.g., \cite[Theorem 1.5]{Schultz 1973}).
Moreover, if $V_k^1(J)$ is the P1-finite element space, then it can be taken by $C_J(k)=\frac{k}{\pi}$ (see, e.g., \cite[Theorem 2.4]{Schultz 1973}).

From \cite[Lemma 2.2]{Kinoshita 2011}, if $V_k^1(J)$ is P1-finite element space (i.e., the basis functions $\psi_i$ are piecewise linear functions), then $P^k_1$ coincides with $\Pi_k$. For any element  $u \in V(\Omega,J)$, we define the semidiscrete projection $P_hu \in V^1\bigl(J;S_h(\Omega)\bigr)$ by the following weak form:
\begin{align}
  \Inner{\Dev{}{t}\bigl(u(t)-P_hu(t)\bigr)}{v_h}{L^2(\Omega)} + \nu\Inner{\nabla \bigl(u(t)-P_hu(t)\bigr)}{\nabla v_h}{L^2(\Omega)^d} = 0, \label{eq:IPP:Ph def} \\ \quad \forall v_h \in S_h(\Omega),~~\AEW t \in J, \nonumber 
\end{align}
where $\AEW$means an abbreviation for 'almost everywhere'.

Finally, the space $S_h^k(\Omega,J)$ is defined as the tensor product $V_k^1(J) \otimes S_h(\Omega)$, which corresponds to a full discretization.
Moreover, we define the full-discretization operator $P_h^k:V(\Omega,J) \to S_h^k(\Omega,J)$ by $P_h^k:=\Pi_kP_h$.
In addition, we denote the matrix norm induced from the Euclidean 2-norm by $\|\cdot\|_{E}$ and denote the transposed matrix of the matrix ${\bf X}$ by ${\bf X}^{\rm T}$.

We show known results for the equation \Ref{parabolic} below.
%%%%%%%%%%%%%%%%%%%%%%%%%%%%%%%%%%%%%%%%%%%%%%%%%%%%%%%%%%%%
%                         Theorem                          %
%%%%%%%%%%%%%%%%%%%%%%%%%%%%%%%%%%%%%%%%%%%%%%%%%%%%%%%%%%%%
\begin{theorem}(see Theorem 5.5, 5.6, and proof of Theorem 4.6 in \cite{Nakao-SIAM})\label{error-estimate-base}
Let $u\in V(\Omega,J) \cap L^2\bigl(J;X(\Omega)\bigr)$ be a solution of \Ref{parabolic} for $f\in L^2\bigl(J;X(\Omega)\bigr)$.
Then, we have the following estimations.
\begin{eqnarray}
  \Norm{u-P_h^ku}{L^2\bigl(J;H_0^1(\Omega)\bigr)}
    &\leq& C_1(h,k)\Norm{f}{L^2\bigl(J;L^2(\Omega)\bigr)},\label{C1hk}\\
  \Norm{u-P_h^ku}{L^2\bigl(J;L^2(\Omega)\bigr)}
    &\leq& C_0(h,k)\Norm{f}{L^2\bigl(J;L^2(\Omega)\bigr)},\label{C0hk}\\
  \Norm{u(T)-P_hu(T)}{L^2(\Omega)}
    &\leq& c_0(h)\Norm{f}{L^2\bigl(J;L^2(\Omega)\bigr)},\label{c0hk}
\end{eqnarray}
where $C_1(h,k):=\frac{2}{\nu}C_\Omega(h) + \Cinv(h)C_J(k)$, $C_0(h,k)=\frac{8}{\nu}C_\Omega(h)^2+C_J(k)$ and $c_0(h)=\sqrt{\frac{8}{\nu}}C_\Omega(h)$.
\end{theorem}
% ---------------------------------------
\subsection{The full-discrete finite element method}
% ---------------------------------------
We define the bi-linear form $a_0(\cdot,\cdot)$ by
\begin{eqnarray*}
a_0(\phi,\psi)&:=&
\Inner{\Dev{}{t}\phi}{\Dev{}{t}\psi}{L^2\bigl(J;L^2(\Omega)\bigr)}
+\nu\Inner{\nabla\phi}{\Dev{}{t}\nabla\psi}{L^2\bigl(J;L^2(\Omega)\bigr)^d},
\end{eqnarray*}
for $\phi, \psi\in V(\Omega,J)$.
Then, for any element  $u \in V(\Omega,J)$, we define the full-discrete projection $Q_h^ku \in S_h^k(\Omega,J)$ by the following weak form:
\begin{eqnarray*}
  a_0(u-Q_h^ku,v_h^k) = 0, \quad \forall v_h^k \in S_h^k(\Omega,J).
\end{eqnarray*}

First, we define the full-discrete finite element approximation $Q_h^ku \in S_h^k(\Omega,J)$ for \Ref{parabolic} by
\begin{align}
  a_0(Q_h^ku,v_h^k) = \Inner{f}{\Dev{}{t}v_h^k}{L^2\bigl(J;L^2(\Omega)\bigr)}, \quad \forall v_h^k \in S_h^k(\Omega,J).\label{appQ} 
\end{align}
We now have the following estimation from above definition.
Note that the scheme in \cite{Nakao-SIAM} is based on the finite element Galerkin method with an interpolation in time that uses the fundamental solution for semidiscretization in space.
Since, in the derivation procedure, it uses the fundamental matrix of solution for ODEs associated with the semidiscrete approximation, it is necessary to implement the complicated verified computations on matrix functions.
But the present scheme by \Ref{appQ}  need not any such kind of procedures at all.
%%%%%%%%%%%%%%%%%%%%%%%%%%%%%%%%%%%%%%%%%%%%%%%%%%%%%%%%%%%%
%                         Lemma                          %
%%%%%%%%%%%%%%%%%%%%%%%%%%%%%%%%%%%%%%%%%%%%%%%%%%%%%%%%%%%%
\begin{lemma}
Let $u\in V(\Omega,J) \cap L^2\bigl(J;X(\Omega)\bigr)$ be a solution of \Ref{parabolic} for $f\in L^2\bigl(J;L^2(\Omega)\bigr)$.
Then, the full-discrete projection $Q_h^ku \in S_h^k(\Omega,J)$ satisfies the following $V^1$-stability: 
\begin{eqnarray*}
  \Norm{Q_h^ku}{V^1\bigl(J;L^2(\Omega)\bigr)}
    &\leq&\Norm{f}{L^2\bigl(J;L^2(\Omega)\bigr)}.
\end{eqnarray*}
\end{lemma}
\begin{proof}
From \Ref{appQ} and $Q_h^ku(x,0)=0$ in $\Omega$, if we take $v_h^k=Q_h^ku$ then it follows that
\begin{eqnarray*}
\Norm{Q_h^ku}{V^1\bigl(J;L^2(\Omega)\bigr)}^2 + \frac{\nu}{2}\Norm{Q_h^ku(T)}{H^1_0(\Omega)}^2
&=&
a_0(Q_h^ku,Q_h^ku)\\
&=&   
\Inner{f}{\Dev{}{t}Q_h^ku}{L^2\bigl(J;L^2(\Omega)\bigr)}\\
&\le&   
\Norm{f}{L^2\bigl(J;L^2(\Omega)\bigr)}\Norm{Q_h^ku}{V^1\bigl(J;L^2(\Omega)\bigr)}.
\end{eqnarray*}
Therefore, the proof is completed.
\end{proof}

Next, we consider the estimation $\Norm{Q_h^ku}{L^2\bigl(J;H^1_0(\Omega)\bigr)}$ for $u \in V(\Omega,J) \cap L^2\bigl(J;X(\Omega)\bigr)$.
We now define $\alpha_h^k\in S_h^k(\Omega,J)$ satisfying $\Inner{\Dev{}{t}\alpha_h^k}{\Dev{}{t}v_h^k}{L^2\bigl(J;L^2(\Omega)\bigr)}=\Inner{f}{\Dev{}{t}v_h^k}{L^2\bigl(J;L^2(\Omega)\bigr)}$ for  all $v_h^k \in S_h^k(\Omega,J)$.
Note that by taking $v_h^k=\alpha_h^k$, it follows that $\Norm{\alpha_h^k}{V^1\bigl(J;L^2(\Omega)\bigr)}\le\Norm{f}{L^2\bigl(J;L^2(\Omega)\bigr)}$, which also implies the unique existence of $\alpha_h^k$.
We now define the matrices ${\bf A}$ and ${\bf M}$ in $\mathbb{R}^{mn \times mn}$ by 
$$
{\bf A}_{i,j}:=\Inner{\Dev{}{t}\varphi_j}{\Dev{}{t}\varphi_i}{L^2\bigl(J;L^2(\Omega)\bigr)},\quad 
{\bf M}_{i,j}:=\Inner{\nabla\varphi_j}{\nabla\varphi_i}{L^2\bigl(J;L^2(\Omega)\bigr)^d},\quad
\forall i,j \in \{1,\ldots,mn\},
$$
respectively.
Since matrices ${\bf A}$ and ${\bf M}$ are symmetric and positive definite, we can denote the Cholesky decomposition as ${\bf A}={\bf A}^{\frac{1}{2}}{\bf A}^{\frac{\rm T}{2}}$ and ${\bf M}={\bf M}^{\frac{1}{2}}{\bf M}^{\frac{\rm T}{2}}$, respectively.
Moreover, we define the matrix ${\bf B}$ in $\mathbb{R}^{mn \times mn}$ by
$$
{\bf B}_{i,j}:=\Inner{\nabla\varphi_j}{\Dev{}{t}\nabla\varphi_i}{L^2\bigl(J;L^2(\Omega)\bigr)^d},\quad\forall i,j \in \{1,\ldots,mn\}.
$$
From the fact that $Q_h^ku$ and $\alpha_h^k$ belong to $S_h^k(\Omega,J)$, there exist coefficient vectors $\mathfrak{u}:=(\mathfrak{u}_1,\ldots,\mathfrak{u}_{mn})^{\rm T}$ and $\mathfrak{a}:=(\mathfrak{a}_1,\ldots,\mathfrak{a}_{mn})^{\rm T}$ in $\mathbb{R}^{mn}$ such that
$Q_h^ku= \sum_{i=1}^{mn} \mathfrak{u}_i\varphi_i = \varphi^{\rm T}\mathfrak{u}$ and $\alpha_h^k= \sum_{i=1}^{mn} \mathfrak{a}_i\varphi_i = \varphi^{\rm T}\mathfrak{a}$
where $\varphi := (\varphi_1,\ldots,\varphi_{mn})^{\rm T}$.
Then, the equation \Ref {appQ} is equivalent to the following.
\begin{eqnarray}\label{Mat-eq}
({\bf A}+\nu {\bf B})\mathfrak{u}={\bf A}\mathfrak{a}.
\end{eqnarray}
Thus we have the following result.
%%%%%%%%%%%%%%%%%%%%%%%%%%%%%%%%%%%%%%%%%%%%%%%%%%%%%%%%%%%%
%                         Lemma                          %
%%%%%%%%%%%%%%%%%%%%%%%%%%%%%%%%%%%%%%%%%%%%%%%%%%%%%%%%%%%%
\begin{lemma}
Let $u\in V(\Omega,J) \cap L^2\bigl(J;X(\Omega)\bigr)$ be a solution of \Ref{parabolic} for $f\in L^2\bigl(J;L^2(\Omega)\bigr)$.
Then, the full-discrete projection $Q_h^ku \in S_h^k(\Omega,J)$ is bounded as
\begin{eqnarray*}
  \Norm{Q_h^ku}{L^2\bigl(J;H^1_0(\Omega)\bigr)}
    \leq\eta\Norm{f}{L^2\bigl(J;L^2(\Omega)\bigr)},
\end{eqnarray*}
where $\eta:=\|{\bf M}^{\frac{\rm T}{2}}({\bf A}+\nu {\bf B})^{-1}{\bf A}^{\frac{1}{2}}\|_E$.
\end{lemma}
\begin{proof}
From \Ref{Mat-eq}, we obtain
\begin{eqnarray*}
\Norm{Q_h^ku}{L^2\bigl(J;H^1_0(\Omega)\bigr)}=\|{\bf M}^{\frac{\rm T}{2}}\mathfrak{u}\|_E
&=&\|{\bf M}^{\frac{\rm T}{2}}({\bf A}+\nu {\bf B})^{-1}{\bf A}\mathfrak{a}\|_E\\
&\le&\|{\bf M}^{\frac{\rm T}{2}}({\bf A}+\nu {\bf B})^{-1}{\bf A}^{\frac{1}{2}}\|_E\|{\bf A}^{\frac{\rm T}{2}}\mathfrak{a}\|_E\\
&=&\eta\Norm{\alpha_h^k}{V^1\bigl(J;L^2(\Omega)\bigr)}\\
&\le&\eta\Norm{f}{L^2\bigl(J;L^2(\Omega)\bigr)}.
\end{eqnarray*}
Therefore, the proof is completed.
\end{proof}

Now, in order to compare our scheme with another one presented in \cite{Nakao-JCAM}, we give some arguments below.
By using $S_h^k(\Omega,J)$, we define the rather simple and natural looking full-discrete finite element scheme in both directions for the problem \Ref{parabolic}.
First, we define the bi-linear form $\hat{a}_0(\cdot,\cdot)$ by
\begin{eqnarray*}
\hat{a}_0(\phi,\psi)&:=&
\Inner{\Dev{}{t}\phi}{\psi}{L^2\bigl(J;L^2(\Omega)\bigr)}
+\nu\Inner{\nabla\phi}{\nabla\psi}{L^2\bigl(J;L^2(\Omega)\bigr)^d},
\end{eqnarray*}
for $\phi, \psi\in V(\Omega,J)$.
Then, for any element  $u \in V(\Omega,J)$, we define the full-discrete projection $\hat{Q}_h^ku \in S_h^k(\Omega,J)$ by the following weak form:
\begin{eqnarray*}
\hat{a}_0(u-\hat{Q}_h^ku,v_h^k) = 0, \quad \forall v_h^k \in S_h^k(\Omega,J).
\end{eqnarray*}
Using above, we define the full-discrete finite element approximation $\hat{Q}_h^ku \in S_h^k(\Omega,J)$ for \Ref{parabolic} by
\begin{align}
  \hat{a}_0(\hat{Q}_h^ku,v_h^k) = \Inner{f}{v_h^k}{L^2\bigl(J;L^2(\Omega)\bigr)}, \quad \forall v_h^k \in S_h^k(\Omega,J).\label{hatQ} 
\end{align}

Let $u\in V(\Omega,J) \cap L^2\bigl(J;X(\Omega)\bigr)$ be a solution of \Ref{parabolic} for $f\in L^2\bigl(J;L^2(\Omega)\bigr)$
Then, by taking $v_h^k=\hat{Q}_h^ku$, we can obtain
\begin{eqnarray*}
\frac{1}{2}\Norm{\hat{Q}_h^ku(T)}{L^2(\Omega)}^2 + \nu\Norm{\hat{Q}_h^ku}{L^2\bigl(J;H^1_0(\Omega)\bigr)}^2
&=&
\hat{a}_0(\hat{Q}_h^ku,\hat{Q}_h^ku)\\
&=&
\Inner{f}{\hat{Q}_h^ku}{L^2\bigl(J;L^2(\Omega)\bigr)}\\
&\le&   
\Norm{f}{L^2\bigl(J;L^2(\Omega)\bigr)}\Norm{\hat{Q}_h^ku}{L^2\bigl(J;L^2(\Omega)\bigr)}\\
&\le&   
C_p\Norm{f}{L^2\bigl(J;L^2(\Omega)\bigr)}\Norm{\hat{Q}_h^ku}{L^2\bigl(J;H^1_0(\Omega)\bigr)}.
\end{eqnarray*}
Thus we have the following $L^2H^1_0$-stability:
\begin{eqnarray*}
\Norm{\hat{Q}_h^ku}{L^2\bigl(J;H^1_0(\Omega)\bigr)}
\le   
\frac{C_p}{\nu}\Norm{f}{L^2\bigl(J;L^2(\Omega)\bigr)}.
\end{eqnarray*}
Moreover, we consider the estimation $\Norm{\hat{Q}_h^ku}{V^1\bigl(J;L^2(\Omega)\bigr)}$ for $u \in V(\Omega,J) \cap L^2\bigl(J;X(\Omega)\bigr)$.
Now we define $\hat{\alpha}_h^k\in S_h^k(\Omega,J)$ satisfying $\Inner{\hat{\alpha}_h^k}{v_h^k}{L^2\bigl(J;L^2(\Omega)\bigr)}=\Inner{f}{v_h^k}{L^2\bigl(J;L^2(\Omega)\bigr)}$ for  all $v_h^k \in S_h^k(\Omega,J)$.
Note that by taking $v_h^k=\hat{\alpha}_h^k$, it follows that $\Norm{\hat{\alpha}_h^k}{L^2\bigl(J;L^2(\Omega)\bigr)}\le\Norm{f}{L^2\bigl(J;L^2(\Omega)\bigr)}$.
Moreover, from \Ref{hatQ}, we obtain
\begin{eqnarray}\label{mat-unstable}
  \hat{a}_0(\hat{Q}_h^ku,v_h^k)
=\Inner{\hat{\alpha}_h^k}{v_h^k}{L^2\bigl(J;L^2(\Omega)\bigr)}.
\end{eqnarray}
We now define the matrices ${\bf G}$ and ${\bf U}$ in $\mathbb{R}^{mn \times mn}$ by 
$$
{\bf G}_{i,j}:=\Inner{\Dev{}{t}\varphi_j}{\varphi_i}{L^2\bigl(J;L^2(\Omega)\bigr)},\quad 
{\bf U}_{i,j}:=\Inner{\varphi_j}{\varphi_i}{L^2\bigl(J;L^2(\Omega)\bigr)^d},\quad
\forall i,j \in \{1,\ldots,mn\},
$$
respectively.
Since the matrix ${\bf U}$ is symmetric and positive definite, we denote the Cholesky decomposition as ${\bf U}={\bf U}^{\frac{1}{2}}{\bf U}^{\frac{\rm T}{2}}$.
From the fact that $\hat{Q}_h^ku$ and $\hat{\alpha}_h^k$ in $S_h^k(\Omega,J)$, there exist coefficient vectors $\hat{\mathfrak{u}}:=(\hat{\mathfrak{u}}_1,\ldots,\hat{\mathfrak{u}}_{mn})^{\rm T}$ and $\hat{\mathfrak{a}}:=(\hat{\mathfrak{a}}_1,\ldots,\hat{\mathfrak{a}}_{mn})^{\rm T}$ in $\mathbb{R}^{mn}$ such that
$\hat{Q}_h^ku= \sum_{i=1}^{mn} \hat{\mathfrak{u}}_i\varphi_i = \varphi^{\rm T}\hat{\mathfrak{u}}$ and $\hat{\alpha}_h^k= \sum_{i=1}^{mn} \hat{\mathfrak{a}}_i\varphi_i = \varphi^{\rm T}\hat{\mathfrak{a}}$, respectively.
Then, the variational equation \Ref {mat-unstable} is equivalent to the following.
\begin{eqnarray*}
({\bf G}+\nu {\bf M})\hat{\mathfrak{u}}={\bf U}\hat{\alpha}.
\end{eqnarray*}
Letting $\hat{\eta}:=\|{\bf A}^{\frac{\rm T}{2}}({\bf G}+\nu {\bf M})^{-1}{\bf U}^{\frac{1}{2}}\|_E$, it follows that
\begin{eqnarray*}
\Norm{\hat{Q}_h^ku}{V^1\bigl(J;L^2(\Omega)\bigr)}=\|{\bf A}^{\frac{\rm T}{2}}\hat{\mathfrak{u}}\|_E
&=&\|{\bf A}^{\frac{\rm T}{2}}({\bf G}+\nu {\bf M})^{-1}{\bf U}\hat{\mathfrak{a}}\|_E\\
&\le&\|{\bf A}^{\frac{\rm T}{2}}({\bf G}+\nu {\bf M})^{-1}{\bf U}^{\frac{1}{2}}\|_E\|{\bf U}^{\frac{\rm T}{2}}\hat{\mathfrak{a}}\|_E\\
&=&\hat{\eta}\Norm{\hat{\alpha}_h^k}{L^2\bigl(J;L^2(\Omega)\bigr)}\\
&\le&\hat{\eta}\Norm{f}{L^2\bigl(J;L^2(\Omega)\bigr)},
\end{eqnarray*}
which enables us the $V^1L^2$ estimates for the scheme \Ref{hatQ}.

For $\nu=1$, $\nu=0.1$ and $\nu=0.01$ in $\Omega=(0,1)$ and $J=(0,1)$, Table \ref{table1-1} and \ref{table1-2} show verified results of $\eta$ and $\hat{\eta}$, respectively. 
By the verified computing results, we can conclude that the projection $\hat{Q}_h^k$ is not $V^1$-stable, and our proposed projection $Q_h^k$ satisfies $V^1$-stability as well as it has $L^2H^1_0$-stability.
%%%%%%%%%%%%%%%%%%%%%%%%%%%%%%%%%%%%%%%%%%%%%%%%%%%%%%%%%%%%
\begin{table}[ht]
\begin{center}
\caption{The numerical results $\eta$ in $\Omega=(0,1)$, $J=(0,1)$. }
%\footnotesize
\begin{tabular}{|c||c|c|c||c|c|c||c|c|c||}
\hline
&\multicolumn{3}{|c||}{$\nu=1$}&\multicolumn{3}{|c||}{$\nu=0.1$}&\multicolumn{3}{|c||}{$\nu=0.01$}\\\cline{2-10}
&\multicolumn{3}{|c||}{$h$}&\multicolumn{3}{|c||}{$h$}&\multicolumn{3}{|c||}{$h$}\\
$k$&1/5&1/10&1/20&1/5&1/10&1/20&1/5&1/10&1/20\\
\hline
\hline
%1/20&0.3016&0.3048&0.3056&1.4023&1.3965&1.3950&4.6187&4.6599&4.6525\\
1/40&0.3014&0.3047&0.3055&1.4026&1.3968&1.3953&4.6191&4.6606&4.6532\\
%1/60&0.3014&0.3046&0.3054&1.4027&1.3968&1.3953&4.6192&4.6607&4.6534\\
1/80&0.3014&0.3046&0.3054&1.4027&1.3968&1.3953&4.6192&4.6607&4.6534\\
%1/100&0.3014&0.3046&0.3054&1.4027&1.3969&1.3953&4.6192&4.6607&4.6534\\
1/120&0.3014&0.3046&0.3054&1.4027&1.3969&1.3954&4.6192&4.6608&4.6535\\
%1/140&0.3014&0.3046&0.3054&1.4027&1.3969&1.3954&4.6192&4.6608&4.6535\\
1/160&0.3014&0.3046&0.3054&1.4027&1.3969&1.3954&4.6192&4.6608&4.6535\\
%1/180&0.3014&0.3046&0.3054&1.4027&1.3969&1.3954&4.6192&4.6608&4.6535\\
1/200&0.3014&0.3046&0.3054&1.4027&1.3969&1.3954&4.6192&4.6608&4.6535\\
%1/220&0.3014&0.3046&0.3054&1.4027&1.3969&1.3954&4.6192&4.6608&4.6535\\
1/240&0.3014&0.3046&0.3054&1.4027&1.3969&1.3954&4.6192&4.6608&4.6535\\
%1/260&0.3014&0.3046&0.3054&1.4027&1.3969&1.3954&4.6192&4.6608&4.6535\\
1/280&0.3014&0.3046&0.3054&1.4027&1.3969&1.3954&4.6192&4.6608&4.6535\\
%1/300&0.3014&0.3046&0.3054&1.4027&1.3969&1.3954&4.6193&4.6608&4.6535\\
1/320&0.3014&0.3046&0.3054&1.4028&1.3969&1.3954&4.6193&4.6608&4.6535\\
%1/340&0.3014&0.3046&0.3054&1.4028&1.3969&1.3954&4.6193&4.6608&4.6535\\
1/360&0.3014&0.3046&0.3054&1.4028&1.3969&1.3954&4.6193&4.6608&4.6535\\
%1/380&0.3014&0.3046&0.3054&1.4028&1.3969&1.3954&4.6193&4.6608&4.6535\\
1/400&0.3014&0.3046&0.3054&1.4028&1.3969&1.3954&4.6193&4.6608&4.6536\\
\hline
\end{tabular}
\label{table1-1} 
\caption{The numerical results $\hat{\eta}$ in $\Omega=(0,1)$, $J=(0,1)$. }
%\footnotesize
\begin{tabular}{|c||c|c|c||c|c|c||c|c|c||}
\hline
&\multicolumn{3}{|c||}{$\nu=1$}&\multicolumn{3}{|c||}{$\nu=0.1$}&\multicolumn{3}{|c||}{$\nu=0.01$}\\\cline{2-10}
&\multicolumn{3}{|c||}{$h$}&\multicolumn{3}{|c||}{$h$}&\multicolumn{3}{|c||}{$h$}\\
$k$&1/5&1/10&1/20&1/5&1/10&1/20&1/5&1/10&1/20\\
\hline
\hline
%1/20&5.44&5.54&5.57&12.88&12.93&12.94&14.52&14.53&14.53\\
1/40&10.92&11.12&11.16&25.75&25.83&25.85&29.01&29.03&29.03\\
%1/60&16.39&16.68&16.75&38.62&38.75&38.78&43.51&43.53&43.53\\
1/80&21.86&22.25&22.34&51.49&51.66&51.69&58.01&58.03&58.04\\
%1/100&27.33&27.81&27.93&64.36&64.57&64.62&72.51&72.54&72.54\\
1/120&32.80&33.37&33.52&77.24&77.48&77.54&87.01&87.04&87.05\\
%1/140&38.27&38.94&39.10&90.11&90.39&90.47&101.51&101.55&101.56\\
1/160&43.74&44.50&44.69&102.98&103.31&103.39&116.01&116.05&116.06\\
%1/180&49.20&50.06&50.28&115.86&116.22&116.31&130.51&130.56&130.57\\
1/200&54.67&55.63&55.87&128.73&129.14&129.23&145.01&145.07&145.08\\
%1/220&60.14&61.19&61.45&141.6&142.05&142.16&159.52&159.57&159.59\\
1/240&65.61&66.75&67.04&154.48&154.96&155.08&174.02&174.08&174.09\\
%1/260&71.08&72.32&72.63&167.35&167.88&168.00&188.52&188.59&188.60\\
1/280&76.55&77.88&78.21&180.22&180.79&180.93&203.02&203.09&203.11\\
%1/300&82.01&83.44&83.80&193.10&193.70&193.85&217.52&217.6&217.62\\
1/320&87.48&89.00&89.39&205.97&206.61&206.77&232.02&232.11&232.12\\
%1/340&92.95&94.57&94.97&218.84&219.53&219.7&246.53&246.61&246.63\\
1/360&98.42&100.13&100.56&231.72&232.44&232.62&261.03&261.12&261.14\\
%1/380&103.89&105.69&106.15&244.59&245.35&245.54&275.53&275.62&275.65\\
1/400&109.35&111.26&111.73&257.46&258.27&258.47&290.03&290.13&290.15\\
\hline
\end{tabular}
\label{table1-2} 
\end{center}
\end{table}  
%%%%%%%%%%%%%%%%%%%%%%%%%%%%%%%%%%%%%%%%%%%%%%%%%%%%%%%%%%%%
%                         Remark                          %
%%%%%%%%%%%%%%%%%%%%%%%%%%%%%%%%%%%%%%%%%%%%%%%%%%%%%%%%%%%%
\begin{remark}\label{re:constant}
All computations in Tables are carried out on the Dell Precision 5820 Intel Xeon CPU 4.0GHz by using INTLAB, a tool box in MATLAB developed by Rump \cite{Rump INTLAB} for self-validating algorithms. Therefore, all numerical values in these tables are verified data in the sense of strictly rounding error control.
Moreover, we take the basis of finite element subspaces $S_h(\Omega)$ and $V_k^1(J)$ are taken as P1-function with uniform mesh on $\Omega$ and $J$, respectively. 
%The numerical value shown on the table cuts off the numerical value by which the precision was guaranteed by the suitable effective number.
\end{remark}
% ---------------------------------------
\setcounter{equation}{0}
\section{Constructive error estimates}
% ---------------------------------------
In this section, we consider a constructive error estimates of the projection $Q_h^k$ for the finite element approximation.
For an arbitrary  $u \in V(\Omega,J) \cap L^2\bigl(J;X(\Omega)\bigr)$, we define the projection $\bar{P}_h^ku\in S_h^k(\Omega,J)$ satisfying the following weak form:
\begin{eqnarray}\label{barPhk}
\Inner{\Dev{}{t}\bar{P}_h^ku}{\Dev{}{t}v_h^k}{L^2\bigl(J;L^2(\Omega)\bigr)}+\nu\Inner{\nabla P_h u}{\Dev{}{t}\nabla v_h^k}{L^2\bigl(J;L^2(\Omega)\bigr)^d}
=\Inner{\Delta_tu}{\Dev{}{t}v_h^k}{L^2\bigl(J;L^2(\Omega)\bigr)},
\end{eqnarray}
for all $v_h^k\in S_h^k(\Omega,J)$. 
Note that, for a fixed $v_h^k\in S_h^k(\Omega,J)$, by taking $v_h$ as $v_h=\Dev{}{t}v_h^k$ in \Ref{eq:IPP:Ph def}, we have
\begin{eqnarray}\label{extend-Phk}
  \Inner{\Dev{}{t}P_hu}{\Dev{}{t}v_h^k}{L^2\bigl(J;L^2(\Omega)\bigr)} + \nu\Inner{\nabla P_hu}{\Dev{}{t}v_h^k}{L^2\bigl(J;L^2(\Omega)\bigr)^d} = \Inner{\Delta_tu}{\Dev{}{t}v_h^k}{L^2\bigl(J;L^2(\Omega)\bigr)}. 
\end{eqnarray}
From \Ref{barPhk} and \Ref{extend-Phk}, we obtain
\begin{eqnarray}\label{P1}
  \Inner{\Dev{}{t}P_hu}{\Dev{}{t}v_h^k}{L^2\bigl(J;L^2(\Omega)\bigr)}= \Inner{\Dev{}{t}\bar{P}_h^ku}{\Dev{}{t}v_h^k}{L^2\bigl(J;L^2(\Omega)\bigr)}. 
\end{eqnarray}
Moreover, from the definition of $V^1$-projection, we have
\begin{eqnarray}\label{P2}
  \Inner{\Dev{}{t}P_hu}{\Dev{}{t}v_h^k}{L^2\bigl(J;L^2(\Omega)\bigr)}= \Inner{\Dev{}{t}P_1^kP_hu}{\Dev{}{t}v_h^k}{L^2\bigl(J;L^2(\Omega)\bigr)}. 
\end{eqnarray}
From \Ref{P1} and \Ref{P2}, it follows that $\bar{P}_h^k=P_1^kP_h$ because $P_1^kP_hu\in S_h^k(\Omega,J)$.
%%%%%%%%%%%%%%%%%%%%%%%%%%%%%%%%%%%%%%%%%%%%%%%%%%%%%%%%%%%%
%                         Remark                          %
%%%%%%%%%%%%%%%%%%%%%%%%%%%%%%%%%%%%%%%%%%%%%%%%%%%%%%%%%%%%
\begin{remark}\cite[Lemma 2.2]{Kinoshita 2011}\label{re:p1}
If $V_k^1(J)$ is the  P1-finite element space, then $P^k_1$ coincides with $\Pi_k$, it follows that $\bar{P}_h^k=P_h^k(=\Pi_kP_h)$.
\end{remark}
For the projection $Q_h^k$, from the triangle inequality, we have
\begin{eqnarray}
  \Norm{u-Q_h^ku}{L^2\bigl(J;H_0^1(\Omega)\bigr)}
    &\leq&\Norm{u-P_h^ku}{L^2\bigl(J;H_0^1(\Omega)\bigr)}+\Norm{P_h^ku-Q_h^ku}{L^2\bigl(J;H_0^1(\Omega)\bigr)},\label{tC1hk}\\
  \Norm{u-Q_h^ku}{L^2\bigl(J;L^2(\Omega)\bigr)}
    &\leq&\Norm{u-P_h^ku}{L^2\bigl(J;L^2(\Omega)\bigr)}+\Norm{P_h^ku-Q_h^ku}{L^2\bigl(J;L^2(\Omega)\bigr)},\label{tC0hk}\\
  \Norm{u(T)-Q_h^ku(T)}{L^2(\Omega)}
    &\leq&\Norm{u(T)-P_hu(T)}{L^2(\Omega)}+\Norm{P_h^ku(T)-Q_h^ku(T)}{L^2(\Omega)},\label{tc0hk}
\end{eqnarray}
where we have used the fact that $P_h^ku(T)=\Pi_kP_hu(T)=P_hu(T)$.
Thus we now present the estimation for $P_h^ku-Q_h^ku\in S_h^k(\Omega,J)$ below.

From \Ref{appQ} and \Ref{barPhk}, and letting $\delta_h^k:=\bar{P}_h^ku-Q_h^ku\in S_h^k(\Omega,J)$, we can obtain
\begin{eqnarray}\label{second-error}
a_0(\delta_h^k,v_h^k)
=\nu\Inner{\nabla \xi}{\Dev{}{t}\nabla v_h^k}{L^2\bigl(J;L^2(\Omega)\bigr)^d}
\end{eqnarray}
where $\xi:=\bar{P}_h^ku-P_h u\in V$.
Moreover, we define $\beta_h^k\in S_h^k(\Omega,J)$ satisfying 
$$
\Inner{\Dev{}{t}\nabla \beta_h^k}{\Dev{}{t}\nabla v_h^k}{L^2\bigl(J;L^2(\Omega)\bigr)^d}=\Inner{\nabla \xi}{\Dev{}{t}\nabla v_h^k}{L^2\bigl(J;L^2(\Omega)\bigr)^d},\quad  \forall v_h^k \in S_h^k(\Omega,J).
$$
Note that by taking $v_h^k=\beta_h^k$ on the above, it follows that $\Norm{\Dev{}{t}\beta_h^k}{L^2\bigl(J;H^1_0(\Omega)\bigr)}\le\Norm{\xi}{L^2\bigl(J;H^1_0(\Omega)\bigr)}$.
Then we obtain
\begin{eqnarray}\label{mat-error}
a_0(\delta_h^k,v_h^k)
=\nu\Inner{\Dev{}{t}\nabla \beta_h^k}{\Dev{}{t}\nabla v_h^k}{L^2\bigl(J;L^2(\Omega)\bigr)^d}
\end{eqnarray}
We now define the matrices ${\bf W}$ and ${\bf Y}$ in $\mathbb{R}^{mn \times mn}$ by 
$$
{\bf W}_{i,j}:=\Inner{\Dev{}{t}\nabla\varphi_j}{\Dev{}{t}\nabla\varphi_i}{L^2\bigl(J;L^2(\Omega)\bigr)^d},\quad 
{\bf Y}_{i,j}:=\Inner{\varphi_j(\cdot,T)}{\varphi_i(\cdot,T)}{L^2(\Omega)},\quad
\forall i,j \in \{1,\ldots,mn\},
$$
respectively.
Since the matrix ${\bf W}$ is symmetric and positive definite, we can denote the Cholesky decomposition as ${\bf W}={\bf W}^{\frac{1}{2}}{\bf W}^{\frac{\rm T}{2}}$.
Moreover, since the matrix ${\bf Y}$ is symmetric and positive semi-definite, we can decompose it as ${\bf Y}={\bf Y}^{\frac{1}{2}}{\bf Y}^{\frac{\rm T}{2}}$.
From the fact that $\delta_h^k$ and $\beta_h^k$ in $S_h^k(\Omega,J)$, there exist coefficient vectors $\mathfrak{d}:=(\mathfrak{d}_1,\ldots,\mathfrak{d}_{mn})^{\rm T}$ and $\mathfrak{b}:=(\mathfrak{b}_1,\ldots,\mathfrak{b}_{mn})^{\rm T}$ in $\mathbb{R}^{mn}$ such that
$\delta_h^k= \sum_{i=1}^{mn} \mathfrak{d}_i\varphi_i = \varphi^{\rm T}\mathfrak{d}$ and $\beta_h^k= \sum_{i=1}^{mn} \mathfrak{b}_i\varphi_i = \varphi^{\rm T}\mathfrak{b}$.
Then, the variational equation \Ref {mat-error} is equivalent to the following.
\begin{eqnarray}\label{Mat-e}
({\bf A}+\nu {\bf B})\mathfrak{d}=\nu {\bf W}\mathfrak{b}.
\end{eqnarray}

Let
\begin{eqnarray*}
\gamma_1&:=&\nu\|{\bf M}^{\frac{\rm T}{2}}({\bf A}+\nu {\bf B})^{-1}{\bf W}^{\frac{1}{2}}\|_E,\\
\gamma_0&:=&\nu\|{\bf U}^{\frac{\rm T}{2}}({\bf A}+\nu {\bf B})^{-1}{\bf W}^{\frac{1}{2}}\|_E,\\
\gamma_T&:=&\nu\|{\bf Y}^{\frac{\rm T}{2}}({\bf A}+\nu {\bf B})^{-1}{\bf W}^{\frac{1}{2}}\|_E.
\end{eqnarray*}
Then we have the following main result in this paper.
%%%%%%%%%%%%%%%%%%%%%%%%%%%%%%%%%%%%%%%%%%%%%%%%%%%%%%%%%%%%
%                         Theorem                          %
%%%%%%%%%%%%%%%%%%%%%%%%%%%%%%%%%%%%%%%%%%%%%%%%%%%%%%%%%%%%
\begin{theorem}\label{thm:error}
Assume that $V_k^1(J)$ is the  P1 finite element space.
Let $u\in V(\Omega,J) \cap L^2\bigl(J;X(\Omega)\bigr)$ be a solution of \Ref{parabolic} for $f\in L^2\bigl(J;L^2(\Omega)\bigr)$.
Then, we have the following estimations.
\begin{eqnarray*}
  \Norm{u-Q_h^ku}{L^2\bigl(J;H_0^1(\Omega)\bigr)}
    &\leq& \tilde{C}_1(h,k)\Norm{f}{L^2\bigl(J;L^2(\Omega)\bigr)},\\
  \Norm{u-Q_h^ku}{L^2\bigl(J;L^2(\Omega)\bigr)}
    &\leq& \tilde{C}_0(h,k)\Norm{f}{L^2\bigl(J;L^2(\Omega)\bigr)},\\
  \Norm{u(T)-Q_h^ku(T)}{L^2(\Omega)}
    &\leq& \tilde{c}_0(h)\Norm{f}{L^2\bigl(J;L^2(\Omega)\bigr)},
\end{eqnarray*}
where 
\begin{eqnarray*}
\tilde{C}_1(h,k)&\equiv&C_1(h,k)+C_J(k)C_{inv}(h)\gamma_1,\\
\tilde{C}_0(h,k)&\equiv&C_0(h,k)+C_J(k)C_{inv}(h)\gamma_0,\\
\tilde{c}_0(h,k)&\equiv&c_0(h)+C_J(k)C_{inv}(h)\gamma_T.
\end{eqnarray*}
\end{theorem}
\begin{proof}
From \Ref{Mat-e}, we can obtain
\begin{eqnarray*}
\Norm{\delta_h^k}{L^2\bigl(J;H^1_0(\Omega)\bigr)}
&=&\|{\bf M}^{\frac{\rm T}{2}}\mathfrak{d}\|_E
=\nu\|{\bf M}^{\frac{\rm T}{2}}({\bf A}+\nu {\bf B})^{-1}{\bf W}\mathfrak{b}\|_E
\le\gamma_1\|{\bf W}^{\frac{\rm T}{2}}\mathfrak{b}\|_E\\
\Norm{\delta_h^k}{L^2\bigl(J;L^2(\Omega)\bigr)}
&=&\|{\bf U}^{\frac{\rm T}{2}}\mathfrak{d}\|_E
=\nu\|{\bf U}^{\frac{\rm T}{2}}({\bf A}+\nu {\bf B})^{-1}{\bf W}\mathfrak{b}\|_E
\le\gamma_0\|{\bf W}^{\frac{\rm T}{2}}\mathfrak{b}\|_E,\\
\Norm{\delta_h^k(T)}{L^2(\Omega)}
&=&\|{\bf Y}^{\frac{\rm T}{2}}\mathfrak{d}\|_E
=\nu\|{\bf Y}^{\frac{\rm T}{2}}({\bf A}+\nu {\bf B})^{-1}{\bf W}\mathfrak{b}\|_E
\le\gamma_T\|{\bf W}^{\frac{\rm T}{2}}\mathfrak{b}\|_E.
\end{eqnarray*}
Moreover, we have
$$
\|{\bf W}^{\frac{\rm T}{2}}\mathfrak{b}\|_E=\Norm{\Dev{}{t}\beta_h^k}{L^2\bigl(J;H^1_0(\Omega)\bigr)}\le\Norm{\xi}{L^2\bigl(J;H^1_0(\Omega)\bigr)}.
$$
Note that $\bar{P}_h^k=P_h^k(=\Pi_kP_h)$ from Remark \ref{re:p1}. Then it follows that
\begin{eqnarray*}
\Norm{\xi}{L^2\bigl(J;H^1_0(\Omega)\bigr)}
=\Norm{\bar{P}_h^ku-P_h u}{L^2\bigl(J;H^1_0(\Omega)\bigr)}
&=&\Norm{\Pi_kP_hu-P_h u}{L^2\bigl(J;H^1_0(\Omega)\bigr)}\\
&\le&C_{inv}(h)\Norm{\Pi_kP_hu-P_h u}{L^2\bigl(J;L^2(\Omega)\bigr)}\\
&\le&C_J(k)C_{inv}(h)\Norm{P_hu}{V^1\bigl(J;L^2(\Omega)\bigr)}\\
&\le&C_J(k)C_{inv}(h)\Norm{f}{L^2\bigl(J;L^2(\Omega)\bigr)},
\end{eqnarray*}
where we have used the fact that $\Norm{P_hu}{V^1\bigl(J;L^2(\Omega)\bigr)}\le\Norm{f}{L^2\bigl(J;L^2(\Omega)\bigr)}$ in \cite{Nakao-JCAM}.
Therefore, the proof is completed from \Ref{tC1hk}, \Ref{tC0hk}, \Ref{tc0hk}, Theorem \ref{error-estimate-base} and the fact $\delta_h^k=P_h^ku-Q_h^ku$.
\end{proof}

The same assumptions in Remark \ref{re:constant}, Table \ref{table2-1}, \ref{table2-2}  and \ref{table2-3} show verified results of  $\gamma_1$, $\gamma_0$ and $\gamma_T$ for $\nu=1$, $\nu=0.1$ and $\nu=0.01$ in $\Omega=(0,1)$ and $J=(0,1)$.
From the verified results in Table \ref{table2-1}, \ref{table2-2}  and \ref{table2-3},  we may conclude that $\gamma_0$ and $\gamma_T$ are dependent on the parameter $\nu$ more clearly than $\gamma_1$, but asymptotically converge to some fixed constants when $h$ and $k$ tend to zero.
%%%%%%%%%%%%%%%%%%%%%%%%%%%%%%%%%%%%%%%%%%%%%%%%%%%%%%%%%%%%
\begin{table}[ht]
\caption{The numericai results in $\Omega=(0,1)$, $J=(0,1)$. }
\footnotesize
\begin{tabular}{|c||c|c|c||c|c|c||c|c|c||}
\hline
$\nu=1$&\multicolumn{3}{|c||}{$h=1/5$}&\multicolumn{3}{|c||}{$h=1/10$}&\multicolumn{3}{|c||}{$h=1/20$}\\
\hline
$k$&$\gamma_1$&$\gamma_0$&$\gamma_T$&$\gamma_1$&$\gamma_0$&$\gamma_T$&$\gamma_1$&$\gamma_0$&$\gamma_T$\\
\hline
\hline
%1/20&3.0766&0.3016&0.7071&8.9359&0.3048&0.7071&12.8983&0.3056&0.7071\\
1/40&1.6381&0.3014&0.7071&7.2951&0.3047&0.7071&18.2519&0.3055&0.7071\\
%1/60&1.0959&0.3014&0.7071&5.2410&0.3046&0.7071&17.6586&0.3054&0.7071\\
1/80&0.9999&0.3014&0.7071&3.9947&0.3046&0.7071&15.2511&0.3054&0.7071\\
%1/100&0.9999&0.3014&0.7071&3.2110&0.3046&0.7071&12.9125&0.3054&0.7071\\
1/120&0.9999&0.3014&0.7071&2.6805&0.3046&0.7071&11.0305&0.3054&0.7071\\
%1/140&0.9999&0.3014&0.7071&2.2993&0.3046&0.7071&9.5675&0.3054&0.7071\\
1/160&0.9999&0.3014&0.7071&2.0126&0.3046&0.7071&8.4231&0.3054&0.7071\\
%1/180&0.9999&0.3014&0.7071&1.7893&0.3046&0.7071&7.5128&0.3054&0.7071\\
1/200&0.9999&0.3014&0.7071&1.6105&0.3046&0.7071&6.7750&0.3054&0.7071\\
%1/220&0.9999&0.3014&0.7071&1.4642&0.3046&0.7071&6.1667&0.3054&0.7071\\
1/240&0.9999&0.3014&0.7071&1.3422&0.3046&0.7071&5.6573&0.3054&0.7071\\
%1/260&0.9999&0.3014&0.7071&1.2390&0.3046&0.7071&5.2249&0.3054&0.7071\\
1/280&0.9999&0.3014&0.7071&1.1505&0.3046&0.7071&4.8534&0.3054&0.7071\\
%1/300&0.9999&0.3014&0.7071&1.0738&0.3046&0.7071&4.5310&0.3054&0.7071\\
1/320&0.9999&0.3014&0.7071&1.0067&0.3046&0.7071&4.2486&0.3054&0.7071\\
%1/340&0.9999&0.3014&0.7071&1.0000&0.3046&0.7071&3.9993&0.3054&0.7071\\
1/360&0.9999&0.3014&0.7071&1.0000&0.3046&0.7071&3.7775&0.3054&0.7071\\
%1/380&0.9999&0.3014&0.7071&1.0000&0.3046&0.7071&3.5790&0.3054&0.7072\\
1/400&0.9999&0.3014&0.7071&1.0000&0.3046&0.7071&3.4003&0.3054&0.7072\\
\hline
\end{tabular}
\\
$
(k,\gamma_1)=(1/500,2.7210), 
%(k,\gamma_1)=(1/600,2.2680), 
(k,\gamma_1)=(1/700,1.9451),
%(k,\gamma_1)=(1/800,1.7034), 
(k,\gamma_1)=(1/900,1.5170)
$
for $h=1/20$.
\label{table2-1}
\caption{The numericai results in $\Omega=(0,1)$, $J=(0,1)$. }
\begin{tabular}{|c||c|c|c||c|c|c||c|c|c||}
\hline
$\nu=0.1$&\multicolumn{3}{|c||}{$h=1/5$}&\multicolumn{3}{|c||}{$h=1/10$}&\multicolumn{3}{|c||}{$h=1/20$}\\
\hline
$k$&$\gamma_1$&$\gamma_0$&$\gamma_T$&$\gamma_1$&$\gamma_0$&$\gamma_T$&$\gamma_1$&$\gamma_0$&$\gamma_T$\\
\hline
\hline
%1/20&0.9922&0.1402&0.2236&1.5896&0.1396&0.2236&5.4926&0.1395&0.2236\\
1/40&0.9915&0.1402&0.2236&0.9998&0.1396&0.2236&3.3302&0.1395&0.2236\\
%1/60&0.9914&0.1402&0.2236&0.9997&0.1396&0.2236&2.2581&0.1395&0.2236\\
1/80&0.9914&0.1402&0.2236&0.9996&0.1396&0.2236&1.6986&0.1395&0.2236\\
%1/100&0.9914&0.1402&0.2236&0.9996&0.1396&0.2236&1.3599&0.1395&0.2236\\
1/120&0.9914&0.1402&0.2236&0.9996&0.1396&0.2236&1.1335&0.1395&0.2236\\
%1/140&0.9914&0.1402&0.2236&0.9996&0.1396&0.2236&1.0000&0.1395&0.2236\\
1/160&0.9914&0.1402&0.2236&0.9996&0.1396&0.2236&0.9999&0.1395&0.2236\\
%1/180&0.9914&0.1402&0.2236&0.9996&0.1396&0.2236&0.9999&0.1395&0.2236\\
1/200&0.9913&0.1402&0.2236&0.9996&0.1396&0.2236&0.9999&0.1395&0.2236\\
%1/220&0.9913&0.1402&0.2236&0.9996&0.1396&0.2236&1.0000&0.1395&0.2236\\
1/240&0.9913&0.1402&0.2236&0.9996&0.1396&0.2236&1.0000&0.1395&0.2236\\
%1/260&0.9913&0.1402&0.2236&0.9996&0.1396&0.2236&1.0000&0.1395&0.2236\\
1/280&0.9913&0.1402&0.2236&0.9996&0.1396&0.2236&1.0000&0.1395&0.2236\\
%1/300&0.9913&0.1402&0.2236&0.9996&0.1396&0.2236&1.0000&0.1395&0.2236\\
1/320&0.9913&0.1402&0.2236&0.9996&0.1396&0.2236&1.0000&0.1395&0.2236\\
%1/340&0.9913&0.1402&0.2236&0.9996&0.1396&0.2236&1.0000&0.1395&0.2236\\
1/360&0.9914&0.1402&0.2236&0.9996&0.1396&0.2236&1.0000&0.1395&0.2236\\
%1/380&0.9914&0.1402&0.2236&0.9996&0.1396&0.2236&1.0001&0.1395&0.2236\\
1/400&0.9914&0.1402&0.2236&0.9996&0.1396&0.2236&1.0001&0.1395&0.2236\\
\hline
\end{tabular}
\label{table2-2} 
\caption{The numericai results in $\Omega=(0,1)$, $J=(0,1)$. }
\begin{tabular}{|c||c|c|c||c|c|c||c|c|c||}
\hline
$\nu=0.01$&\multicolumn{3}{|c||}{$h=1/5$}&\multicolumn{3}{|c||}{$h=1/10$}&\multicolumn{3}{|c||}{$h=1/20$}\\
\hline
$k$&$\gamma_1$&$\gamma_0$&$\gamma_T$&$\gamma_1$&$\gamma_0$&$\gamma_T$&$\gamma_1$&$\gamma_0$&$\gamma_T$\\
\hline
\hline
%1/20&0.6971&0.0461&0.0697&0.9687&0.0465&0.0707&0.9988&0.0465&0.0707\\
1/40&0.6972&0.0461&0.0697&0.9682&0.0466&0.0707&0.9981&0.0465&0.0707\\
%1/60&0.6972&0.0461&0.0697&0.9681&0.0466&0.0707&0.9979&0.0465&0.0707\\
1/80&0.6972&0.0461&0.0697&0.9681&0.0466&0.0707&0.9979&0.0465&0.0707\\
%1/100&0.6972&0.0461&0.0697&0.9681&0.0460&0.0707&0.9979&0.0463&0.0707\\
1/120&0.6972&0.0461&0.0697&0.9681&0.0466&0.0707&0.9979&0.0465&0.0707\\
%1/140&0.6972&0.0461&0.0697&0.9681&0.0466&0.0707&0.9978&0.0465&0.0707\\
1/160&0.6972&0.0461&0.0697&0.9681&0.0466&0.0707&0.9978&0.0465&0.0707\\
%1/180&0.6972&0.0461&0.0697&0.9681&0.0466&0.0707&0.9978&0.0465&0.0707\\
1/200&0.6972&0.0461&0.0697&0.9681&0.0466&0.0707&0.9978&0.0465&0.0707\\
%1/220&0.6972&0.0461&0.0697&0.9681&0.0466&0.0707&0.9978&0.0465&0.0707\\
1/240&0.6972&0.0461&0.0697&0.9681&0.0466&0.0707&0.9978&0.0465&0.0707\\
%1/260&0.6972&0.0461&0.0697&0.9681&0.0466&0.0707&0.9978&0.0465&0.0707\\
1/280&0.6972&0.0461&0.0697&0.9681&0.0466&0.0707&0.9978&0.0465&0.0707\\
%1/300&0.6972&0.0461&0.0697&0.9681&0.0466&0.0707&0.9978&0.0465&0.0707\\
1/320&0.6972&0.0461&0.0697&0.9681&0.0466&0.0707&0.9978&0.0465&0.0707\\
%1/340&0.6972&0.0461&0.0697&0.9681&0.0466&0.0707&0.9978&0.0465&0.0707\\
1/360&0.6972&0.0461&0.0697&0.9681&0.0466&0.0707&0.9979&0.0465&0.0707\\
%1/380&0.6972&0.0461&0.0697&0.9681&0.0466&0.0707&0.9979&0.0465&0.0707\\
1/400&0.6972&0.0461&0.0697&0.9681&0.0466&0.0707&0.9979&0.0465&0.0707\\
\hline
\end{tabular}
\label{table2-3} 
\end{table}  
% ---------------------------------------
\setcounter{equation}{0}
\section{Conclusion}
% ---------------------------------------
We presented a new full-discrete finite element projection $Q_h^k$ for the heat equation, and derived the constructive stability by the numerical computations with guaranteed accuracy.
Our scheme is closely related to that in \cite{Nakao-SIAM} in the sense that the constructive error estimates is established by using the results obtained by the same paper.
Therefore, it is considered as an another version of \cite{Nakao-SIAM}, but the present scheme should be more familiar method to researchers working on numerical analysis.
Namely, it is not necessary any complicated manipulation for the verified computation of matrix function.% such as in the previous paper.

Particularly, the estimate $\Norm{u(T)-Q_h^ku(T)}{L^2(\Omega)}$ should be useful to the verified computation for nonlinear problems by the time-evolutional method (cf.\cite{Nakao-NM}), which will be presented in our forthcoming paper.
Thus, our method will play an important role in the numerical verification method to find exact solutions for the nonlinear parabolic equations.
%%%%%%%%%%%%%%%%%%%%%%%%%%%%%%%%%%%%%%%%%%%%%%%%%%%%%%%%%%%%%%%%%%%%%%%%%%

\end{document}